\documentclass[12pt,letterpaper]{article}
\usepackage{amsmath,amsfonts,amsthm,amssymb}

\swapnumbers 

\newtheorem{lemma}{Lemma}[section]
\newtheorem{corollary}[lemma]{Corollary}
\newtheorem{theorem}[lemma]{Theorem}
\newtheorem{assumptions}[lemma]{Standing Assumptions}

\theoremstyle{definition} 
\newtheorem{definition}[lemma]{Definition}
\newtheorem{remark}[lemma]{Remark}

\newcommand{\Nat}{{\mathbb N}}

\newcommand\reals{{\mathbb R}}

\newcommand{\dg}{\sp{\text{\rm o}}}
\newcommand\bul{\noindent$\bullet\ $}

\begin{document}

\title{Banach synaptic algebras }

\author{David J. Foulis{\footnote{Emeritus Professor, Department of
Mathematics and Statistics, University of Massachusetts, Amherst,
MA; Postal Address: 1 Sutton Court, Amherst, MA 01002, USA;
foulis@math.umass.edu.}}\hspace{.05 in} and Sylvia
Pulmannov\'{a}{\footnote{ Mathematical Institute, Slovak Academy of
Sciences, \v Stef\'anikova 49, SK-814 73 Bratislava, Slovakia;
pulmann@mat.savba.sk. The second  author was supported by grant
VEGA No.2/0069/16.}}}

\date{}

\maketitle

\begin{abstract}
\noindent Using a representation theorem of Erik Alfsen, Frederic Schultz, and
Erling St{\o}rmer for special JB-algebras, we prove that a synaptic algebra
is norm complete (i.e., Banach) if and only if it is isomorphic to the
self-adjoint part of a Rickart C$\sp{\ast}$-algebra. Also, we give conditions
on a Banach synaptic algebra that are equivalent to the condition that it is
isomorphic to the self-adjoint part of an AW$\sp{\ast}$-algebra. Moreover, we
study some relationships between synaptic algebras and so-called generalized
Hermitian algebras.
\end{abstract}

\section{Introduction}

A synaptic algebra \cite{FSyn, FPproj, FPtype, FPsym, FPcom, FJP2proj, FJPpande,
vectlat, FJPstat, FJPMSR, Pid} (see Section \ref{sc:SAs} below) is a generalization
of several structures based on operator algebras. The adjective `synaptic' is
derived from the Greek word `sunaptein,' meaning to join together; indeed
synaptic algebras unite the notions of an order-unit normed space \cite
[p. 69]{Alf}, a real unital special Jordan algebra \cite{McC}, a convex
effect algebra \cite{GPBB}, and an orthomodular lattice \cite{Beran, Kalm}.

Important examples of synaptic algebras include the JW-algebras, the AJW-algebras,
and the spin factors of D. Topping \cite{Top65, FPSpin}, as well as the
ordered special Jordan algebras of T. Sarymsakov \emph{et al.} \cite{Sary}.
Also the generalized Hermitian algebras introduced and studied in
\cite{GHAlg1, GHAlg2} are synaptic algebras. Moreover, the self-adjoint
parts of von Neumann algebras \cite{Sakai}, Rickart C$\sp{\ast}$-algebras
\cite[\S 3]{HHL}, and AW$\sp{\ast}$-algebras \cite{Kap} are synaptic algebras.
Whereas most of the preceeding examples are Banach (i.e., norm-complete)
algebras, a synaptic algebra need not be norm-complete (e.g., \cite[Example
1.2]{FSyn}).

We shall be using results of Erik Alfsen, Frederic Schultz, and Erling
St{\o}rmer in \cite{AlfSchSto} pertaining to so-called JB-algebras to prove
that a synaptic algebra is isomorphic to a Rickart JC-algebra (i.e., the
self-adjoint part of a Rickart C$\sp{\ast}$-algebra of bounded linear
operators on a Hilbert space) if and only if it is a Banach algebra.

In Section \ref{sc:prelim} we review some basic definitions pertaining to
$\ast$-algebras, we recall the definition and some of the relevant properties
of a JB-algebra in Section \ref{sc:JB}, and in Section \ref{sc:SAs}, we present
the axioms SA1--SA8 for a synaptic algebra and remark on the extent to which
these axioms hold for a JB-algebra. In Theorem \ref{th:BanachSA}, we prove that
a Banach synaptic algebra is a special JB-algebra.  In Section \ref{sc:RepTh},
we relate the carrier and Rickart properties and prove that a synaptic algebra
is Banach if and only if it is isomorphic to the self-adjoint part of a Rickart
C$\sp{\ast}$-algebra (Theorem \ref{th:RBSA}). In Section \ref{sc:Additional},
we present some additional properties of a synaptic algebra, in Section \ref
{sc:blocks} we review some facts about blocks and C-blocks in a synaptic
algebra, and in Section \ref{sc:AW*}, we characterize Banach synaptic algebras
that are isomorphic to the self-adjoint part of an AW$\sp{\ast}$-algebra.
Finally, in Section \ref{sc:SA&GH}, we investigate some of the relationships
between synaptic algebras and so-called generalized Hermitian (GH) algebras.

\section{Preliminaries} \label{sc:prelim}

In this paper we use `iff' as an abbreviation for `if and only if,' the
notation `$:=$' means `equals by definition,' $\reals$ is the ordered field
of real numbers, and $\Nat:=\{1,2,3,...\}$ is the well-ordered set of
natural numbers.

A Jordan algebra $J$ \cite{McC} with Jordan product $\odot$ is \emph{unital}
iff there is a \emph{unit element} $1\in J$ such that $a\odot 1=a$ for all
$a\in J$. The Jordan algebra $J$ is \emph{special} iff it can be embedded
in an associative algebra $R$ in such a way that $a\odot b=\frac12(ab+ba)$
for all $a,b\in J$, where $ab+ba$ is calculated in $R$ \cite{McC}.

Let ${\mathcal C}$ be an associative algebra over the complex numbers. If
there is a unit element $1\in{\mathcal C}$ such that $1c=c1=c$ for all $c\in
{\mathcal C}$, then ${\mathcal C}$ is said to be \emph{unital}. If there is
an ``adjoint mapping" $c\mapsto c\sp{\ast}$ on ${\mathcal C}$ such that, for
all $c,d\in {\mathcal C}$ and for every complex number $\alpha$, (1) $(\alpha c)
\sp{\ast}={\bar{\alpha}}c\sp{\ast}$, (2) $(c+d)\sp{\ast}=c\sp{\ast}+d\sp{\ast}$,
(3) $(cd)\sp{\ast}=d\sp{\ast}c\sp{\ast}$, and (4) $(c\sp{\ast})\sp{\ast}=c$,
then ${\mathcal C}$ is called a \emph{$\ast$-algebra}. An \emph{isomorphism}
from one $\ast$-algebra to another is understood to be an algebra isomorphism
that preserves the adjoint mapping. If the $\ast$-algebras are unital, then
such an isomorphism automatically preserves $1$.

Let ${\mathcal C}$ be a $\ast$-algebra. Then an element $c\in{\mathcal C}$
is said to be \emph{self-adjoint} iff $c=c\sp{\ast}$, and a self-adjoint
idempotent $p=p\sp{\ast}=p\sp{2}\in{\mathcal C}$ is called a \emph{projection}.
We denote by ${\mathcal C}\sp{sa}$ the set of all self-adjoint elements in
${\mathcal C}$ and we refer to ${\mathcal C}\sp{sa}$ as the \emph{self-adjoint
part} of ${\mathcal C}$. Clearly, if $c\in{\mathcal C}$, then $c+c\sp{\ast}$,
$\frac{1}{i}(c-c\sp{\ast})$, $cc\sp{\ast}, c\sp{\ast}c\in{\mathcal C}\sp{sa}$,
and if ${\mathcal C}$ is unital, then $1\in{\mathcal C}\sp{sa}$. Moreover,
${\mathcal C}\sp{sa}$ is a special real Jordan algebra under the Jordan product
$c\odot d:=\frac12(cd+dc)$ for all $c,d\in{\mathcal C}\sp{sa}$. Also, if
${\mathcal C}$ is unital, then the Jordan algebra ${\mathcal C}\sp{sa}$ is
unital.

An (abstract) unital C$\sp{\ast}$-\emph{algebra} is defined to be a unital
Banach ${\ast}$-algebra ${\mathcal C}$ such that the norm satisfies
$\|cc\sp{\ast}\|=\|c\|\sp{2}$ for all $c\in{\mathcal C}$. In what follows,
\emph{we shall consider only unital C$\sp{\ast}$ algebras; thus we shall
omit the adjective ``unital."} An \emph{isomorphism} of one
C$\sp{\ast}$-algebra onto another is understood to be an isomorphism of
${\ast}$-algebras that is also an isometry.

\begin{lemma} \label{lm:cc*}
Let ${\mathcal C}$ be a C$\sp{\ast}$-algebra, let $c\in{\mathcal C}$, and
let $p=p\sp{2}$ be a projection in ${\mathcal C}$. Then{\rm: (i)} $cc
\sp{\ast}=0\Leftrightarrow c=0$. {\rm (ii)} $cc\sp{\ast}=pcc\sp{\ast}
\Leftrightarrow c=pc$.
\end{lemma}

\begin{proof}
(i) $cc\sp{\ast}=0\Leftrightarrow 0=\|cc\sp{\ast}\|=\|c\|\sp{2}
\Leftrightarrow c=0$. (ii) Suppose that $cc\sp{\ast}=pcc\sp{\ast}$.
Then $(1-p)p=0$, so $(1-p)c[(1-p)c]\sp{\ast}=(1-p)cc\sp{\ast}(1-p)=
(1-p)pcc\sp{\ast}(1-p)=0$, whence by (i), $(1-p)c=0$, i.e., $c=pc$.
The converse is obvious.
\end{proof}

If ${\mathfrak H}$ is a complex Hilbert space then ${\mathcal B}({\mathfrak H})$
denotes the unital $\ast$-algebra of all bounded linear operators on ${\mathfrak H}$
under the formation of operator adjoints and with the uniform operator norm.
A norm-closed $\ast$-subalgebra ${\mathcal C}$ of ${\mathcal B}({\mathfrak H})$
with $1\in{\mathcal C}$ can be shown to be a C$\sp{\ast}$-algebra and as such,
it is called a \emph{concrete} C$\sp{\ast}$-\emph{algebra}. By a classical result
of I. Gelfand and M. Neumark \cite{GandN}, every C$\sp{\ast}$-algebra is isomorphic
to a concrete C$\sp{\ast}$-algebra.

If ${\mathcal C}$ is a $C\sp{\ast}$-algebra, then the self-adjoint part
${\mathcal C}\sp{sa}$ of ${\mathcal C}$ is not only a special real unital Jordan
algebra, but it turns out to be a JB-algebra as per Section \ref{sc:JB} below.
By definition, a \emph{JC-algebra} is the self-adjoint part ${\mathcal C}\sp{sa}$
of a concrete ${\mathcal C}\sp{\ast}$-algebra, i.e., it is a norm closed
unital special Jordan algebra of self-adjoint operators on a complex Hilbert space
\cite{Top65}. Thus, a JC-algebra is also a JB-algebra.

Both JB-algebras and synaptic algebras are so-called order-unit normed spaces,
according to the following definition.

\begin{definition} [{\cite[pp. 67--69]{Alf}}] \label{df:OUNS}
An \emph{order-unit space}  is a real partially ordered vector space $V$
with a positive cone $V\sp{+}:=\{v\in V:0\leq v\}$ and with a distinguished
element $1\in V\sp{+}$, called the \emph{order unit} such that:
\begin{enumerate}
\item[(1)] For every $v\in V$, there exists $n\in\Nat$ such that $v\leq n1$.
\end{enumerate}
An \emph{order-unit normed space} with order unit $1$ is defined to be an
order-unit space $V$ with order unit $1$ such that:
\begin{enumerate}
\item[(2)] $V$ is \emph{archimedean}, i.e., if $v,w\in V$ and $nv\leq w$
 for all $n\in\Nat$, then $-v\in V\sp{+}$ (equivalently, $v\leq 0$).
\item[(3)] The \emph{order-unit norm} $\|\cdot\|$ on $V$ is defined by
 $\|v\| :=\inf\{\lambda\in\reals:0<\lambda\text{\ and\ }-\lambda1\leq v
 \leq\lambda1\}$.
\end{enumerate}
\end{definition}

\section{JB algebras} \label{sc:JB}

In this section we review the definition, some notation, and some facts
pertaining to JB-algebras as per \cite{AlfSchSto}.

\begin{definition} [{\cite[page 13]{AlfSchSto}}] \label{df:JBalg}
A \emph{JB-algebra} $B$ is both an order-unit space with order unit $1\in B$
and a (not necessarily special) unital Jordan algebra over $\reals$.  Moreover,
$B$ is a Banach space under a norm $\|\cdot\|$ that satisfies the following
conditions for all $b,c\in B$:
\begin{enumerate}
\item[(1)] $\|b\odot c\|\leq\|b\|\|c\|$.
\item[(2)] $\|b\sp{2}\|=\|b\|\sp{2}$. (Note: $b\sp{2}:=b\odot b$.)
\item[(3)] $\|b\sp{2}\|\leq\|b\sp{2}+c\sp{2}\|$.
\end{enumerate}
\end{definition}

\begin{theorem} [{\cite[Theorem 2.1]{AlfSchSto}}] \label{th:JB}
\ \newline
\indent Let $B$ be a JB-algebra. Then{\rm:}
\begin{enumerate}
\item $B$ is a {\rm(}not necessarily special{\rm)} real Jordan algebra
 with  Jordan product $\odot$, and $b\odot 1=b$ for all $b\in B$.
\item $B$ is an order-unit normed space with order unit $1$ and the
 given norm $\|\cdot\|$ on $B$ is the order-unit norm.
\item The positive cone in $B$ satisfies $B\sp{+}:=\{b\in B:0\leq b\}
 =\{b\sp{2}:b\in B\}$.
\item Under $\|\cdot\|$, $B$ is a norm-complete {\rm(}i.e., a Banach{\rm)}
 space.
\item For all $b\in B$, $-1\leq b\leq 1\Rightarrow 0\leq b\sp{2}\leq 1$.
\end{enumerate}

Conversely, if $B$ is a norm-complete order-unit normed space with order unit
$1$ and also a {\rm(}not necessarily special{\rm)} real unital Jordan algebra,
and if condition {\rm(v)} above holds, then $B$ is a JB-algebra under the
order-unit norm.
\end{theorem}

An \emph{isomorphism} of one JB-algebra onto another is understood to be
(1) an order, (2) a linear, and (3) a Jordan isomorphism. In view of (1)
and (2), a JB-algebra isomorphism is an isometry.

The self-adjoint part ${\mathcal C}\sp{sa}$ of a C$\sp{\ast}$-algebra
${\mathcal C}$ is organized into a special JB-algebra as follows: ${\mathcal C}
\sp{sa}$ hosts the special Jordan product $c\odot d:=\frac12(cd+dc)$ for $c,d
\in{\mathcal C}\sp{sa}$ and the positive cone is given by $({\mathcal C}
\sp{sa})\sp{+}=\{cc\sp{\ast}:c\in{\mathcal C}\}$.

\begin{definition} [{\cite[Equation (2.24)]{AlfSchSto}}] \label{df:Ub}
If $B$ is a JB-algebra and $b\in B$, then the mapping $U\sb{b}\colon B
\to B$ is defined by $U\sb{b}c:=2b\odot(b\odot c)-b\sp{2}\odot c$ for
all $c\in B$.
\end{definition}

Clearly $U\sb{b}$ is linear on $B$ and it turns out that $U\sb{b}$ is
positive, hence order preserving on $B$ \cite[Proposition 2.7]{AlfSchSto}.
If $B$ is a special Jordan algebra, then $U\sb{b}c=bcb$ for all $b,c\in B$.

The commutative unital Banach algebra in the next definition plays an
important role in \cite{AlfSchSto} where it is written as $C(b)$ rather
than as $\Gamma(b)$ \cite[page 14]{AlfSchSto}. Here we have changed the
notation to avoid confusion with the notion of ``commutant" in a synaptic
algebra.

\begin{definition} \label{df:GammaB}
If $B$ is a JB-algebra and $b\in B$, then $\Gamma(b)$ is the commutative
unital Banach algebra obtained by forming the norm closure of the associative
Jordan subalgebra of $B$ consisting of all real polynomials in $b$.
\end{definition}

A key result is \cite[Proposition 2.3]{AlfSchSto}, which provides a
functional representation theorem asserting that there is an order and
algebraic isomorphism of $\Gamma(b)$ onto the partially ordered commutative
real Banach algebra $C(X,\reals)$, under pointwise partial order and
operations, of all continuous real-valued functions on a compact Hausdorff
space $X$.

\section{Synaptic algebras} \label{sc:SAs}

To help fix ideas in the following definition, the reader may think of $R$
as the C$\sp{\ast}$-algebra ${\mathcal B}({\mathfrak H})$ of all bounded
linear operators on a complex Hilbert space $\mathfrak H$ and of $A$ as the
self-adjoint part ${\mathcal B}\sp{\rm sa}({\mathfrak H})$ of ${\mathcal B}
({\mathfrak H})$.

\begin{definition} [{\cite[Definition 1.1]{FSyn}}] \label{df:SynapticAlgebra}
Let $R$ be a  real or complex linear associative algebra with unity
element $1$ and let $A$ be a real vector subspace of $R$. If $a,b
\in A$ and $M\subseteq A$, we write $aCb$ iff $a$ and $b$ commute
(i.e. $ab=ba$ as calculated in $R$) and we define
\[
C(a):=\{b\in A: aCb\},\ C(M):=\bigcap\sb{b\in M}C(b),
\]
\[
\ CC(M):=C(C(M),\text{\ and\ } CC(a):=C(C(a)).
\]
We call $C(M)$ the \emph{commutant} of $M$ and $CC(M)$ the \emph{bicommutant}
of $M$. The vector space $A$ is a \emph{synaptic algebra} with \emph{enveloping
algebra} $R$ iff the following conditions are satisfied:
\begin{enumerate}
\item[SA1.] $A$ is an order-unit normed space with order unit $1$, positive
 cone $A\sp{+}=\{a\in A:0\leq a\}$, and $\|\cdot\|$ is the corresponding
 order-unit norm.
\item[SA2.] If $a\in A$ then $a\sp{2}\in A\sp{+}$.
\item[SA3.] If $a,b\in A\sp{+}$, then $aba\in A\sp{+}$.
\item[SA4.] If $a\in A$ and $b\in A\sp{+}$, then $aba=0\Rightarrow
 ab=ba=0$.
\item[SA5.] If $a\in A\sp{+}$, there exists $b\in A\sp{+}\cap CC(a)$
 such that $b\sp{2}=a$.
\item[SA6.] If $a\in A$, there exists $p\in A$ such that $p=p\sp{2}$ and,
 for all $b\in A$, $ab=0\Leftrightarrow pb=0$.
\item[SA7.] If $1\leq a\in A$, there exists $b\in A$ such that $ab=ba=1$.
\item[SA8.] If $a,b\in A$, $a\sb{1}\leq a\sb{2}\leq a\sb{3}\leq\cdots$
 is an  ascending sequence of pairwise commuting elements of $C(b)$
 and $\lim\sb{n\rightarrow\infty}\|a-a\sb{n}\|=0$, then $a\in C(b)$.
\end{enumerate}
\end{definition}

\begin{assumptions} In what follows, we assume that $A$ is a synaptic
algebra with enveloping algebra $R$ and with unit $1$. We assume that
$1\not=0$, i.e., $A\not=\{0\}$. The ``unit interval" in $A$ is denoted
by $E:=\{e\in A:0\leq e\leq 1\}$ and elements in $E$ are called ``effects."
The idempotents in $A$ are called ``projections" and the set of all
projections in $A$ is denoted by $P:=\{p\in A:p=p\sp{2}\}$.
\end{assumptions}

The set $E$ of effects is organized into a convex effect algebra
\cite{FandB, GPBB} under the partial binary operation provided by the
restriction of the addition operation on $A$. Under the restriction of
the partial order on $A$, the set $P :=\{p\in A:p=p\sp{2}\}$ of projections
in $A$ is an orthomodular lattice (OML) \cite{Beran, Kalm} with
orthocomplementation $p\mapsto p\sp{\perp}:=1-p$ \cite[\S 5]{FSyn}.

\

Below we comment briefly on each of the axioms SA1--SA8 and we compare
and contrast SA1--SA8 with features of a JB-algebra $B$.

\

\bul By SA1 and Theorem \ref{th:JB}, both $A$ and $B$ form order-unit
normed spaces with order units $1$ and with order unit norms $\|\cdot\|$.

\

\bul By SA2, $A$ forms a special real Jordan algebra under the Jordan product
\[
a\odot b:=\frac12(ab+ba)=\frac12[(a+b)\sp{2}-a\sp{2}-b\sp{2}]=\frac14[(a+b)
 \sp{2}-(a-b)\sp{2}].
\]
Clearly, $a\in A\Rightarrow a\odot 1=a$, so $A$ is a unital Jordan algebra.
As in Definition \ref{df:Ub}, for each $a\in A$, we define the \emph{quadratic
mapping} $U\sb{a}\colon A\to A$ by $U\sb{a}b:=2a\odot b-a\sp{2}\odot b=aba$.
(In \cite{FSyn}, $aba$ is written as $J\sb{a}b$; here we use $U\sb{a}b$ for
consistency with \cite{AlfSchSto}.) Obviously, $U\sb{a}$ is linear and by
\cite[Theorem 4.2]{FSyn} it is order preserving. Let $a,b,c\in A$. Then $aCb
\Rightarrow ab=ba=a\odot b\in A$. As $a\sp{2}\in A$ and $aCa\sp{2}$, it follows
that $a\sp{3}=a\odot a\sp{2}\in A$, and by induction, $a\sp{n}\in A$ for all
$n\in\Nat$. Consequently, $A$ is closed under the formation of real polynomials
in $a$. Thus, as in Definition \ref{df:GammaB}, we define $\Gamma(a)$ to be the
norm closure of the set of all real polynomials in $a$. It can be shown (see
below) that $\Gamma(a)\subseteq CC(a)$; hence, if $b\in B$, then $\Gamma(b)$
can be regarded as an analogue in $B$ of the bicommutant in $A$.

\

\bul Axiom SA3 specifies that, for $a,b\in A\sp{+}$, $U\sb{a}b\in A\sp{+}$.
But, as mentioned above, we actually have the stronger condition $U\sb{a}b
\in A\sp{+}$ for all $a\in A$ and all $b\in A\sp{+}$. Likewise, as mentioned
above, if $b\in B$ and $c\in B\sp{+}$, then $U\sb{b}c\in B\sp{+}$.

\

\bul Axiom SA4 can be written as $U\sb{a}b=0\Rightarrow ab=ba=a\odot b=0$
for all $a\in A$ and all $b\in A\sp{+}$. Analogously, by \cite[Proposition 2.8]
{AlfSchSto}, $B$ has the weaker property that if $b,c\in B\sp{+}$, then
$U\sb{b}c=0\Rightarrow b\odot c=0$.

\

\bul By SA5, if $a\in A\sp{+}$, then there exists $b\in A\sp{+}\cap CC(a)$
such that $b\sp{2}=a$, and by \cite[Theorem 2.2]{FSyn}, $b$ is uniquely
determined by $a$. Of course, we write $a\sp{1/2}:=b$ and refer to
$a\sp{1/2}$ as the \emph{square root} of $a$. As a consequence of SA2
and SA5, the positive cone in $A$ is given by $A\sp{+}:=\{a\in A:0\leq a\}
=\{a\sp{2}:a\in A\}$. The \emph{absolute value} and the \emph{positive}
and \emph{negative} parts of an element $a\in A$ are defined by $|a|:=
(a\sp{2})\sp{1/2}$, $a\sp{+}:=\frac12(|a|+a)$, and $a\sp{-}:=\frac12
(|a|-a)$, respectively. Then $a=a\sp{+}-a\sp{-}$, $|a|=a\sp{+}+a\sp{-}$,
and $a\sp{+}a\sp{-}=0$. Analogously, by \cite[Equations (2.9) and (2.10]
{AlfSchSto}, $B$ has the property that if $b\in B\sp{+}$, there exists
$c\in\Gamma(b)$ such that $c\sp{2}=b$. (Curiously, there is no indication
that the ``square root" $c$ of $b$ can be chosen to be in $B\sp{+}$.)

\

\bul By SA6, if $a\in A$, there exists a projection $p\in P$ such that,
for all $b\in A$, $ab=0\Leftrightarrow pb=0$, and by \cite[Theorem 2.7]
{FSyn}, $p$ is uniquely determined by $a$. We define $a\dg:=p$ and refer
to $a\dg$ as the \emph{carrier} of $a$.  By \cite[Theorem 2.10 (vi)]{FSyn},
$a\dg\in CC(a)$. Also, for $a,b\in A$, we have $ab=0\Leftrightarrow a\dg b
=0\Leftrightarrow a\dg b\dg=0\Leftrightarrow b\dg a\dg=0\Leftrightarrow
ba=0$. For a self-adjoint linear operator $T$ in ${\mathcal B}
\sp{\rm sa}({\mathfrak H})$, $T\dg$ is the (orthogonal) projection onto
the closure of the range of $T$. In general, the JB-algebra $B$ will fail
to satisfy SA6.

\

\bul By SA7, if $a\in A$ with $1\leq a$, there exists $b\in A$ such
that $ab=ba=1$. Of course, $b$ is called the \emph{inverse} of $a$ and as
usual is written as $a\sp{-1}:=b$. Clearly, $a\sp{-1}$ is uniquely
determined by $a$, and $a\sp{-1}\in CC(a)$. As a consequence of \cite
[Equation (2.21) and Proposition 2.4]{AlfSchSto}, $1\leq b\in B$ implies
that $b$ is invertible in $B$ and the inverse of $b$ belongs to $\Gamma(b)$.

\

\bul Axiom SA8, in the presence of the remaining axioms, turns out to be
equivalent to the condition that, for every $a\in A$, the commutant $C(a)$
is norm-closed in $A$. (see \cite[Theorem 8.11]{FSyn}). Therefore, for $M
\subseteq A$, both $C(M)$ and $CC(M):=C(C(M))$ are norm-closed in $A$. In
particular, the bicommutant $CC(a)$ is norm closed in $A$ and since every
real polynomial in $a$ belongs to $CC(a)$, it follows (as was mentioned above)
that $\Gamma(a)\subseteq CC(a)$. Furthermore, the positive cone $A\sp{+}$
is norm-closed in $A$ \cite[Theorem 4.7 (iii)]{FSyn}. Unless $B$ is a
special Jordan algebra, finding a reasonable analogue of axiom SA8 in $B$
may be problematic.

\

Thus, the synaptic algebra $A$ and the JB-algebra $B$ may differ in the following
important respects: (1) Although $B$ is a Banach algebra, $A$ need not be Banach.
(2) Although $A$ is a special Jordan algebra, $B$ need not be special. (3) $B$
might not satisfy the carrier axiom SA-6. (4) Unless $B$ is a special Jordan
algebra, axiom SA-8 might not have a reasonable interpretation in $B$.

Of course an \emph{isomorphism} (or for emphasis, a \emph{synaptic isomorphism})
of one synaptic algebra onto another, we mean a bijection that is (1) an order,
(2) a linear, and (3) a Jordan isomorphism. Clearly, a synaptic isomorphism is
an isometry.

In the next lemma we collect some properties of the (order-unit) norm in the
synaptic algebra $A$. Note that part (vi) of the lemma corresponds to part
(v) of Theorem \ref{th:JB} for a JB-algebra.

\begin{lemma} \label{lm:NormProps}
Let $a,b,e\in A$. Then{\rm:}
\begin{enumerate}
\item $-\|a\|\leq a\leq\|a\|$.
\item $-b\leq a\leq b\Rightarrow\|a\|\leq\|b\|$.
\item $\|a\odot b\|\leq\|a\|\|b\|$.
\item $\|a\sp{2}\|=\|a\|\sp{2}$.
\item $\|a\sp{2}\|\leq\|a\sp{2}+b\sp{2}\|$.
\item $-1\leq a\leq 1\Leftrightarrow a\sp{2}\in E$.
\item $e\in E\Leftrightarrow e\sp{2}\leq e$.
\end{enumerate}
\end{lemma}

\begin{proof}
(i) Part (i) follows from \cite[Proposition II.1.2]{Alf}.

(ii) Suppose that $-b\leq a\leq b$. By (i), $b\leq\|b\|$,
whence $-\|b\|\leq-b\leq a\leq b\leq\|b\|$, and therefore
$\|a\|\leq\|b\|$.

(iii) See \cite[Lemma 1.7 (iv)]{FSyn}.

(iv) See \cite[Lemma 1.7 (ii)]{FSyn}.

(v) As $0\leq a\sp{2}\leq a\sp{2}+b\sp{2}$, we have
$-(a\sp{2}+b\sp{2})\leq a\sp{2}\leq a\sp{2}+b\sp{2}$, whence
(v) follows from (ii).

(vi) Part (vi) follows from \cite[Lemma 1.7 (i)]{FSyn}.

(vii) By \cite[Lemma 2.5 (i)]{FSyn}, $e\in E\Rightarrow e\sp{2}\leq e$.
Conversely, suppose $e\in A$ with $e\sp{2}\leq e$. Clearly, $0\leq e$.
Also, $0\leq(1-e)\sp{2}=1-2e+e\sp{2}$, whence $0\leq e-e\sp{2}\leq
1-e$, so $0\leq e\leq 1$, i.e., $e\in E$.
\end{proof}

\begin{theorem} \label{th:BanachSA}
If $A$ is a Banach synaptic algebra, then $A$ is a special JB-algebra.
\end{theorem}

\begin{proof}
Suppose that $A$ is Banach. By parts (iii), (iv), and (v) of Lemma
\ref{lm:NormProps}, $A$ satisfies conditions (1), (2), and (3) in
Definition \ref{df:JBalg}.
\end{proof}

\section{Representation of a Banach synaptic algebra as the self-adjoint
part of a Rickart C$\sp{\ast}$-algebra} \label{sc:RepTh}

\begin{definition} \label{df:Rickart}
Let ${\mathcal C}$ be a $C\sp{\ast}$-algebra and let ${\mathcal C}\sp{sa}$
be the self-adjoint part of ${\mathcal C}$. Note that every projection
$p=p\sp{*}=p\sp{2}$ in ${\mathcal C}$ belongs to ${\mathcal C}\sp{sa}$.
\begin{enumerate}
\item[(1)] ${\mathcal C}$ is called a \emph{Rickart} C$\sp{\ast}$-\emph
{algebra} \cite{SaWr} iff, for each $c\in{\mathcal C}$, there is a projection
$c\sp{\prime}\in{\mathcal C}$ such that, for all $d\in{\mathcal C}$, $cd=0
\Leftrightarrow d=c\sp{\prime}d$.
\item[(2)] ${\mathcal C}\sp{sa}$ has the \emph{Rickart property} iff,
for every $b\in{\mathcal C}\sp{sa}$, there exists a projection $p\in
{\mathcal C}\sp{sa}$ such that for all $g\in{\mathcal C}\sp{sa}$, $bg=0
\Leftrightarrow g=pg$. A JC-algebra with the Rickart property is called
a \emph{Rickart} JC-algebra.
\item[(3)] ${\mathcal C}\sp{sa}$ has the \emph{carrier property} iff,
for every $b\in{\mathcal C}\sp{sa}$, there exists a projection $b\dg
\in {\mathcal C}\sp{sa}$ such that for all $g\in{\mathcal C}\sp{sa}$,
$bg=0\Leftrightarrow b\dg g=0$.
\end{enumerate}
\end{definition}

\begin{lemma} \label{lm:Rickart}
Let ${\mathcal C}$ be a $C\sp{\ast}$-algebra. Then the following
conditions are mutually equivalent{\rm: (i)} ${\mathcal C}$ is a
Rickart C$\sp{\ast}$-algebra. {\rm (ii)} ${\mathcal C}\sp{sa}$ has
the Rickart property. {\rm (iii)} ${\mathcal C}\sp{sa}$ has
the carrier property.
\end{lemma}

\begin{proof}
(i) $\Rightarrow$ (ii). Assume (i), let $b,g\in{\mathcal C}\sp{sa}$ and put
$p:=b\sp{\prime}$. Then $p$ is a projection, $p\in{\mathcal C}\sp{sa}$, and
$bg=0\Leftrightarrow g=b\sp{\prime}g\Leftrightarrow g=pg$.

\medskip

(ii) $\Rightarrow$ (iii). Assume (ii), let $b,g\in{\mathcal C}\sp{sa}$, and
put $b\dg:=1-p$. Then $b\dg$ is a projection, $b\dg\in{\mathcal C}\sp{sa}$,
and $bg=0\Leftrightarrow g=pg\Leftrightarrow(1-p)g=0\Leftrightarrow
b\dg g=0$.

\medskip

(iii) $\Rightarrow$ (i). Assume (iii), let $c,d\in{\mathcal C}$, and put
$c\sp{\prime}=1-(c\sp{\ast}c)\dg$. It will suffice to prove that $cd=0
\Leftrightarrow d=c\sp{\prime}d$. We have
\[
cd=0\Rightarrow(c\sp{\ast}c)
 (dd\sp{\ast})=0\Leftrightarrow(c\sp{\ast}c)\dg(dd\sp{\ast})=0
 \Rightarrow[(c\sp{\ast}c)\dg d][d\sp{\ast}(c\sp{\ast}c)\dg]=0\Leftrightarrow
\]
\[
[(c\sp{\ast}c)\dg d][(c\sp{\ast}c)\dg d]\sp{\ast}=0\Leftrightarrow
 (c\sp{\ast}c)\dg d=0\Leftrightarrow d=[1-(c\sp{\ast}c)\dg]d
 \Leftrightarrow d=c\sp{\prime}d.
\]

To prove the converse, we begin by putting $b:=c\sp{\ast}c\in{\mathcal C}
\sp{sa}$, so that $c\sp{\prime}=1-(c\sp{\ast}c)\dg=1-b\dg$. Then by
Definition \ref{df:Rickart} (3) with $g:=c\sp{\prime}=1-b\dg\in{\mathcal C}
\sp{sa}$, we have $b\dg c\sp{\prime}=b\dg(1-b\dg)=b\dg-(b\dg)\sp{2}=b\dg-b\dg
=0$, whence $0=bc\sp{\prime}=b(1-b\dg)=b-bb\dg$, and it follows that $b=bb\dg$,
i.e., $c\sp{\ast}c=c\sp{\ast}c(c\sp{\ast}c)\dg$. Therefore, $c\sp{\ast}c
\left(1-(c\sp{\ast}c)\dg\right)=0$, and we have $c\sp{\ast}cc\sp{\prime}
=0$. Consequently, $0=c\sp{\prime}c\sp{\ast}cc\sp{\prime}=(cc\sp{\prime})\sp
{\ast}(cc\sp{\prime})$, and it follows that $cc\sp{\prime}=0$, whence
$d=c\sp{\prime}d\Rightarrow cd=cc\sp{\prime}d=0$.
\end{proof}

\begin{theorem} \label{th:RBSA}
The synaptic algebra $A$ is norm-complete {\rm(}i.e., Banach{\rm}) iff it is
isomorphic to the self-adjoint part ${\mathcal C}\sp{sa}$ of a Rickart
C$\sp{\ast}$-algebra ${\mathcal C}$.
\end{theorem}

\begin{proof}
Suppose that $A$ is a Banach synaptic algebra. Then, by Theorem \ref{th:BanachSA},
$A$ is a special JB-algebra; hence, by \cite[Lemma 9.4]{AlfSchSto}, $A$ is
isomorphic to a JC-algebra, which by definition is the self-adjoint part
${\mathcal C}\sp{sa}$ of a C$\sp{\ast}$-algebra ${\mathcal C}$ of bounded
linear operators on a complex Hilbert space. The synaptic algebra $A$ has the
carrier property, whence ${\mathcal C}\sp{sa}$ has the carrier property,
so by Lemma \ref{lm:Rickart}, ${\mathcal C}$ is a Rickart C$\sp{\ast}$-algebra.
The converse is proved by a straightforward verification that axioms SA1--SA8
hold for the self-adjoint part of a Rickart C$\sp{\ast}$-algebra.
\end{proof}

\begin{corollary} \label{co:Psigmacomplete}
If $A$ is a Banach synaptic algebra, then $P$ is a $\sigma$-complete
OML.
\end{corollary}

\begin{proof}
By \cite[Theorem 1.8.1]{SKB}, the lattice of projections in a Rickart
C$\sp{\ast}$-algebra is $\sigma$-complete.
\end{proof}

\section{Additional properties of a synaptic algebra} \label{sc:Additional}

We continue to assume that $A$ is a synaptic algebra and that $P$ is the
orthomodular lattice (OML) of projections in $A$. As both $A$ and $P$ are
partially ordered sets (posets for short), we begin by reviewing some
terminology.

Let ${\mathcal P}$ be a poset and let ${\mathcal Q}\subseteq{\mathcal P}$.
A \emph{supremum} (an \emph{infimum}) of ${\mathcal Q}$ in ${\mathcal P}$
is a least upper bound (a greatest lower bound) for $Q$ in ${\mathcal P}$.
Note that the supremum, if it exists (the infimum, if it exists) of the
empty subset of ${\mathcal P}$ is the smallest element (the largest element)
in ${\mathcal P}$. The subset ${\mathcal Q}\subseteq{\mathcal P}$ is
\emph{upward directed} (\emph{downward directed}) iff, for every pair $\{a,b\}
\subseteq{\mathcal Q}$, there exists $c\in{\mathcal Q}$ with $a,b\leq c$
(with $c\leq a,b$).

The poset ${\mathcal P}$ is a \emph{lattice} iff every pair of elements
$\{a,b\}\subseteq{\mathcal P}$ has a \emph{join} (i.e., a supremum)
$a\vee b$ and a \emph{meet} (i.e., an infimum) $a\wedge b$ in ${\mathcal P}$.
Of course, the projections $P$ in $A$ form a lattice; also, if $A$ is
commutative, it too is a lattice (Theorem \ref{th:ComA} below).

A mapping $a\mapsto a\sp{\prime}$ on ${\mathcal P}$ is called an
\emph{involution} iff it is order reversing and of period two, i.e.,
for $a,b\in{\mathcal P}$, $a\leq b\Rightarrow b\sp{\prime}\leq a\sp{\prime}$
and $(a\sp{\prime})\sp{\prime}=a$. An involution $a\mapsto a\sp{\prime}$ on
${\mathcal P}$ provides a ``duality" between existing suprema and infima of
subsets of ${\mathcal P}$ as follows: An element $b\in{\mathcal P}$ is the
supremum ($c\in{\mathcal P}$ is the infimum) of ${\mathcal Q}$ in ${\mathcal P}$
iff $b\sp{\prime}$ is the infimum in ${\mathcal P}$ ($c\sp{\prime}$ is the
supremum in ${\mathcal P}$) of the set $\{q\sp{\prime}:q\in{\mathcal Q}\}$.
For instance, for the poset $A$, $a\mapsto -a$ is an involution. Also, the orthosupplementation mapping $p\mapsto p\sp{\perp}:=1-p$ on the OML $P$ is
an involution on $P$.

Now we assume that the poset ${\mathcal P}$ hosts an involution. Thus
we shall formulate the following conditions in terms of upper bounds and
suprema only, since the dual conditions for lower bounds and infima are
then automatic consequences.

The poset ${\mathcal P}$ is: (1) \emph{$\sigma$-complete} iff every sequence
in ${\mathcal P}$ has a supremum in ${\mathcal P}$, (2) \emph{complete} iff
every subset of ${\mathcal P}$ has a supremum in ${\mathcal P}$, (3) \emph
{Dedekind complete} iff every nonempty subset of ${\mathcal P}$ that has an
upper bound in ${\mathcal P}$ has a supremum in ${\mathcal P}$, (4) \emph
{Dedekind $\sigma$-complete} iff every sequence in ${\mathcal P}$ that is
bounded above in ${\mathcal P}$ has a supremum in ${\mathcal P}$, (5) \emph
{monotone $\sigma$-complete} iff every increasing sequence $a\sb{1}\leq a
\sb{2}\leq a\sb{3}\leq\cdots$ that is bounded above in ${\mathcal P}$ has
a supremum in ${\mathcal P}$, (6) \emph{monotone complete} iff every nonempty
upward directed subset of ${\mathcal P}$ that has an upper bound in
${\mathcal P}$ has a supremum in ${\mathcal P}$.

We now turn our attention back to the synaptic algebra $A$ and the OML
$P$ in $A$. Let $p,q\in P$. If $pCq$, then $p\wedge q=pq$ and $p\vee q=
p+q-pq$. Also, $p\leq q$ iff ($pCq$ and $p=pq=qp=p\wedge q$). The projections
$p$ and $q$ are \emph{orthogonal}, in symbols $p\perp q$, iff $p\leq q\sp
{\perp}$, in which case $q\perp p$, $pCq$, $p\wedge q=pq=qp=0$ and $p\vee q
=p+q$. It can be shown that $pCq$ iff $p=(p\wedge q)\vee(p\wedge q\sp{\perp})$
iff there exist $p\sb{1}, q\sb{1}, d\in P$ such that $p\sb{1}\perp q\sb{1}$,
$(p\sb{1}+q\sb{1})\perp d$, $p\sb{1}+q\sb{1}+d=1$, $p=p\sb{1}+d$, and $q=
q\sb{1}+d$.

The commutant and bicommutant have the following obvious properties.
For all $M,N\subseteq A$: (i) $M\subseteq N\Rightarrow C(N)\subseteq C(M)$.
(ii) $M\subseteq CC(M)$.  From (i) and (ii), it follows that (iii) $M
\subseteq N\Rightarrow CC(M)\subseteq CC(N)$ and (iv) $CC(C(M))=C(M)$. A
subset $T\subseteq A$ is \emph{commutative} iff $aCb$ for all $a,b\in T$,
i.e., iff $T\subseteq C(T)$. If $T$ is commutative, then so is $CC(T)$
and $T\subseteq CC(T)\subseteq C(T)$.

If $D\subseteq P$, then $D$ is an \emph{orthogonal set} iff $p\perp q$
for all $p,q\in D$.  The OML $P$ is \emph{orthocomplete}
(\emph{$\sigma$-orthocomplete}) iff every orthogonal subset (every countable
orthogonal subset) has a supremum in $P$. Clearly, every orthogonal subset
of $P$ is commutative.

\begin{theorem} \label{th:Holland}
{\rm(i)} The OML $P$ is orthocomplete {\rm(}$\sigma$-orthocomplete{\rm)} iff
it is complete {\rm(}$\sigma$-complete{\rm)}. {\rm(ii)} If $D\subseteq P$ is
a maximal orthogonal set, then the supremum of $D$ in $P$ is $1$.
\end{theorem}

\begin{proof}
See \cite{SSH} for the proof of (i). To prove (ii), let $D\subseteq P$ be
a maximal orthogonal set and suppose that $p$ is an upper bound for $D$ in
$P$. Then $d\in D\Rightarrow d\leq p\Rightarrow d\perp p\sp{\perp}$, so
$p\sp{\perp}=1-p\in D$ by maximality. But then, $1-p\leq p$, $1-p=(1-p)p=0$,
i.e., $p=1$. Therefore, $1$ is the only upper bound of $D$ in $P$, so it is
the supremum of $D$ in $P$.
\end{proof}

A subset $S\subseteq A$ is called a \emph{sub-synaptic algebra} of
$A$ iff $S$ is a linear subspace of $A$, $1\in S$, and $S$ is closed
under the formation of squares, square roots, carriers, and inverses.
A sub-synaptic algebra $S$ of $A$ is a synaptic algebra in its own
right under the restrictions to $S$ of the partial order and the
operations on $A$ and with the same enveloping algebra as $A$. For
instance, if $T\subseteq A$, then $C(T)$ is a norm-closed sub-synaptic
algebra of $A$. The commutative norm-closed sub-synaptic algebra $C(A)$
of $A$ is called the \emph{center} of $A$ and the synaptic algebra $A$
is commutative iff $A=C(A)$.

By definition, a \emph{symmetry} in $A$ is an element $s\in A$ such that
$s\sp{2}=1$ \cite{FPsym}. There is a bijective correspondence $s
\leftrightarrow p$ between symmetries $s$ and projections $p$ according
to $p=\frac12(1+s)$ and $s=2p-1$. If $p\in P$, $s:=2p-1$ is the
corresponding symmetry, and $a\in A$, then clearly $aCp\Leftrightarrow aCs
\Leftrightarrow a=sas$.

Spectral theory for $A$ is developed in \cite[\S 8]{FSyn} based on the
notion of the \emph{spectral resolution} of an element $a\in A$, which
is the one-parameter family of projections $\{p\sb{\lambda}:\lambda\in
\reals\}$ defined by $p\sb{\lambda}:=(((a-\lambda)\sp{+})\dg)\sp{\perp}$
for all $\lambda\in\reals$. We shall refer to the projections $p\sb
{\lambda}$, $\lambda\in\reals$, as the \emph{spectral projections} of
(or for) $a$. Note that $p\sb{\lambda}\in CC(a)$ for all $\lambda\in\reals$.
Moreover, for all $b\in A$, $bCa$ iff $bCp\sb{\lambda}$ for every $\lambda
\in\reals$ \cite[Theorem 8.10]{FSyn}. We use the latter fact in the proofs
of the next two theorems.

\begin{theorem}  \label{th:CC(T)}
Let $T\subseteq A$ and suppose that $T$ has a supremum $b$ in $A$. Then
$b\in CC(T)$.
\end{theorem}

\begin{proof}
(i) Let $a\in C(T)$. We have to prove that $bCa$. Let $\{p\sb{\lambda}:
\lambda\in\reals\}$ be the spectral resolution of $a$ and let $t
\in T$. Then $tCp\sb{\lambda}$ for all $\lambda\in\reals$. It will be
sufficient to prove that $bCp\sb{\lambda}$ for all $\lambda\in\reals$.
Let $\lambda\in\reals$ and let $s:=2p\sb{\lambda}-1$ be the
corresponding symmetry. Then $tCs$ for all $t\in T$ and as $t\leq b$,
we have $t=sts=U\sb{s}t\leq U\sb{s}b=sbs$, whence $b\leq sbs$. Applying
the quadratic mapping $U\sb{s}$ again to the latter inequality, we obtain
$sbs\leq s\sp{2}bs\sp{2}=b$, whence $sbs=b$, so $bCs$, and therefore $bCp
\sb{\lambda}$. Consequently, $bCa$.
\end{proof}

Following Kaplansky \cite[p. 237]{Kap}, we may regard the next theorem
as expressing ``continuity properties" of a supremum in the OML $P$.

\begin{theorem} \label{th:SupContinuity}
Let $V\subseteq P$ and suppose that $V$ has a supremum $p$ in the OML
$P$. Then{\rm: (i)} $p\in CC(V)$. {\rm (ii)} If $a\in A$ then $va=0$
for all $v\in V$ iff $pa=0$.
\end{theorem}

\begin{proof}
(i) We cannot use Theorem \ref{th:CC(T)} because $p$ is not necessarily
the supremum of $V$ in $A$. However, a similar argument does work: Let
$a\in C(V)$ and let $s$ be the symmetry corresponding to any spectral
projection of $a$. Then, for all $v\in V$, $vCs$ so $v=svs\leq sps$. But
$(sps)\sp{2}=sps\sp{2}ps=sp\sp{2}s=sps$, so $sps\in P$, and therefore
$p\leq sps$. Hence, $sps\leq p$, so $sps=p$, and consequently $pCs$,
whereupon $pCa$.

(ii) Assume that $a\in A$ and $va=0$ for all $v\in V$. Then, for all
$v\in V$, $va\dg=0$, i.e., $v\leq 1-a\dg\in P$, whence $p\leq 1-a\dg$,
so $pa\dg=0$, and therefore $pa=0$. Conversely, suppose that $pa=0$. If
$v\in V$, then $v\leq p$, whence $v=vp$, and we have $va=vpa=0$.
\end{proof}

The next lemma is the converse of Theorem \ref{th:SupContinuity} (ii).

\begin{lemma} \label{lm:V}
Let $V\subseteq P$ and suppose there exists $p\in P$ such that,
for all $a\in A$, $va=0$ for all $v\in V$ iff $pa=0$. Then $p$
is the supremum of $V$ in $P$.
\end{lemma}

\begin{proof}
Since $p(1-p)=0$, we have $v(1-p)=0$, i.e., $v\leq p$ for all
$v\in V$. Suppose that $v\leq r\in P$ for all $v\in V$. Then, for
all $v\in V$, we have $v(1-r)=0$, so $p(1-r)=0$, i.e., $p\leq r$,
and therefore $p$ is the supremum of $V$ in $P$.
\end{proof}

\begin{theorem} {\rm\cite[Theorem 5.12]{vectlat}} \label{th:ComA}
The following conditions are mutually equivalent{\rm: (i)} $A$ is
commutative. {\rm(ii)} As a partially ordered real linear space, $A$ is
a lattice {\rm(}i.e., $A$ is a vector lattice{\rm)}. {\rm(iii)} The
OML $P$ is a Boolean algebra.
\end{theorem}

\begin{theorem} \label{th:Dedekind&monotone}
Suppose that the synaptic algebra $A$ is commutative. Then{\rm: (i)} $A$
is monotone complete {\rm(}monotone $\sigma$-complete{\rm)} iff $A$ is
Dedekind complete {\rm(}Dedekind $\sigma$-complete{\rm)}. {\rm(ii)} If
$V\subseteq P$ and $a$ is the supremum of $V$ in $A$, then $a\in P$ and
$a$ is the supremum of $V$ in $P$. {\rm(iii)} If $A$ is Dedekind complete
{\rm(}Dedekind $\sigma$-complete{\rm)}, then every subset $V$ of $P$
{\rm(}every sequence $(p\sb{n})\sb{n\in\Nat}$ in $P${\rm)} has a supremum
$p$ in $A$; moreover, $p\in P$, $p$ is also the supremum of $V$ in $P$
{\rm(}$p$ is also the supremum of $(p\sb{n})\sb{n\in\Nat}$ in $P${\rm)},
and $P$ is a complete {\rm(}a $\sigma$-complete{\rm)} Boolean algebra.
\end{theorem}

\begin{proof} (i) Suppose that $A$ is monotone complete and that
the nonempty set $S\subseteq A$ is bounded above in $A$. Since $A$ is
commutative, it is a lattice. Let $T$ be the subset of $S$ obtained by
appending to $S$ all suprema of finite nonempty subsets of $S$. Then
$S$ and $T$ have the same upper bounds in $A$ and $T$ is upward directed,
so it has a supremum $b$ in $A$, whence $b$ is also the supremum of $S$ in
$A$. The converse is obvious. A similar argument holds for the monotone
$\sigma$-complete case. Indeed, for a sequence $(a\sb{n})\sb{n\in\Nat}$
that is bounded above in $A$, the sequence $(b\sb{n})\sb{n\in\Nat}$
defined by $b\sb{n}:=a\sb{1}\vee a\sb{2}\vee\cdots\vee a\sb{n}$ is
monotone increasing and has the same set of upper bounds in $A$ as
$(a\sb{n})\sb{n\in\Nat}$.

(ii) For each $v\in V$, we have $0\leq v\leq a$ and $vCa$, whence by
\cite[Theorem 3.9]{FJPMSR}, $v=v\sp{2}\leq a\sp{2}$ and $v=v\sp{1/2}
\leq a\sp{1/2}$. Therefore, $a\leq a\sp{2}$ and $a\leq a\sp{1/2}$.
Again by \cite[Theorem 3.9]{FJPMSR}, from $a\leq a\sp{1/2}$, we infer
that $a\sp{2}\leq a$, and it follows that $a=a\sp{2}$, so $a\in P$.
Clearly then, $a$ is the supremum of $V$ in $P$.

(iii) Suppose that $A$ is Dedekind complete and that $V\subseteq P$. Then
$V$ is bounded above by $1\in P$, whence it has a supremum $p$ in $A$. Thus,
by (ii), $p\in P$ and $p$ is the supremum of $V$ in $P$. By duality, $V$ has
an infimum in $P$, whence $P$ is a complete Boolean algebra. The proof
in the case that $A$ is Dedekind $\sigma$-complete is similar.
\end{proof}

\begin{theorem} \label{th:ComBanachsigma}
Let $A$ be a commutative Banach synaptic algebra. Then $P$ is a
$\sigma$-complete Boolean algebra and $A$ is both Dedekind and
monotone $\sigma$-complete.
\end{theorem}

\begin{proof}
Assume that $A$ is commutative and Banach. By Corollary \ref
{co:Psigmacomplete} and Theorem \ref{th:ComA}, $P$ is a $\sigma$-complete
Boolean algebra. Let $X$ be the Stone space of $P$, and let $C(X,\reals)$
be the lattice-ordered commutative associative unital Banach algebra under
pointwise partial order and pointwise operations of all continuous real-valued
functions on $X$. By \cite[Theorem 4.1]{FPproj}, there is a norm-dense
subalgebra $F$ of $C(X,\reals)$ such that $F$ is a synaptic algebra, and
there is a synaptic isomorphism $\Psi\colon A\to F$. Since $A$ is
norm-complete, so is $F$, and since $F$ is norm-dense in $C(X,\reals)$,
it follows that $F=C(X,\reals)$. Therefore, $C(X,\reals)$ is a commutative
Banach synaptic algebra, and $\Psi\colon A\to C(X,\reals)$ is a synaptic
isomorphism. By \cite[Theorem 6.3]{FJPstat}, $C(X,\reals)$ is both Dedekind
and monotone $\sigma$-complete; hence $A$ also has these properties.
\end{proof}

\begin{theorem} \label{th:ComBanach}
If $A$ is a commutative Banach synaptic algebra, then the following
conditions are mutually equivalent{\rm: (i)} $P$ is a complete Boolean
algebra. {\rm(ii)} $A$ is Dedekind complete. {\rm(iii)} $A$ is monotone
complete.
\end{theorem}

\begin{proof}
(i) $\Leftrightarrow$ (ii). Suppose that $P$ is a complete Boolean algebra.
Then the Stone space $X$ of $P$ is extremally disconnected, i.e., the
closure of every open subset of $X$ remains open \cite[Chapter 38]{GH}. As
in the proof of Theorem \ref{th:ComBanachsigma}, $C(X,\reals)$ is a Banach
synaptic algebra isomorphic to $A$. By \cite[Theorem 14]{Stone}, $C(X,\reals)$,
hence also $A$, is Dedekind complete, and we have (i) $\Rightarrow$ (ii).
The converse implication follows from Theorem \ref{th:Dedekind&monotone}
(iii).

By Theorem \ref{th:Dedekind&monotone} (i), we have (ii) $\Leftrightarrow$ (iii).
\end{proof}

\section{Blocks and C-blocks in a synaptic algebra} \label{sc:blocks}

A \emph{block} in the OML $P$ is a maximal commutative subset of $P$.
Clearly, every block $Q$ in $P$ is closed under the formation of meets,
joins, and orthocomplements in $P$, and under these operations it is a
Boolean algebra. Evidently, $Q$ is a block in $P$ iff $Q=P\cap C(Q)$. By
Zorn's lemma, every commutative subset of $P$, and in particular, every
singleton subset $\{p\}$ of $P$, can be enlarged to a block in $P$.
Therefore, $P$ is covered by its own blocks.

By analogy with the notion of a block $Q$ in $P$, a maximal commutative
subset $B$ of $A$ is called a \emph{C-block}. Evidently, $B\subseteq A$
is a C-block iff $B=C(B)$, in which case $B=C(B)=CC(B)$. Every
commutative subset $T$ of $A$ can be enlarged to a C-block $B\supseteq T$
(Zorn). In particular, if $a\in A$, then the singleton set $\{a\}$ is
commutative, the bicommutant $CC(a)$ is commutative, $a\in CC(a)$, and
$CC(a)$ can be enlarged to a C-block $B$ with $a\in B$. Therefore, $A$ is
covered by its own C-blocks.

Let $B$ be a C-block in $A$. Then $B=C(B)$ is a commutative norm-closed
sub-synaptic algebra of $A$ and $B\cap P$ is the projection lattice of
$B$. By Theorem \ref{th:ComA}, $B$ is a vector lattice and $B\cap P$ is
a Boolean algebra.  If $p,q\in B\cap P$, then $p\vee q=p+q-pq$ is the
supremum of $p$ and $q$ and $p\wedge q=pq=qp$ is the infimum of $p$ and
$q$ both in $A$ and in the synaptic algebra $B$.

\begin{theorem} \label{th:uniqueblock}
If $Q$ is a block in $P$, then there is a unique C-block $B$ with
$Q\subseteq B$, namely, $B=C(Q)$; moreover, $B=C(B)=CC(B)=C(Q)=CC(Q)$.
Conversely, if $B$ is a C-block in $A$, then $Q:=P\cap B$ is the unique
block in $P$ such that $Q\subseteq B$.
\end{theorem}

\begin{proof}
Since $Q$ is a commutative subset of $A$, there exists a C-block $B$ of
$A$ with $Q\subseteq B$ (Zorn). Let $b\in B$ and let $p\sb{\lambda}$,
$\lambda\in\reals$, be a spectral projection of $b$. Then, for all $q
\in Q\subseteq B$, we have $bCq$, whence $p\sb{\lambda}Cq$, and by the
maximality of $Q$, $p\sb{\lambda}\in Q$. Thus, the spectral resolution
of each element $b\in B$ is contained in $Q$, whence $C(Q)\subseteq C(B)
=CC(B)=B$. Also, since $Q\subseteq B$, we have $B=C(B)\subseteq C(Q)$,
so $B=C(Q)$. Therefore, $C(B)=CC(Q)$, whence $B=C(B)=CC(B)=C(Q)=CC(Q)$.

Conversely, suppose that $B$ is a C-block in $A$ and put $Q:=P\cap B$.
As $Q\subseteq B$, we have $Q\subseteq C(Q)$, so $Q\subseteq P\cap C(Q)$.
Let $b\in B$ and let $p\sb{\lambda}$, $\lambda\in\reals$, be a spectral
projection of $b$. Then $p\sb{\lambda}\in CC(b)\subseteq CC(B)=B$, so the
spectral resolution of $b$ is contained in $P\cap B=Q$. Therefore, if $a
\in C(Q)$, then $a\in C(b)$, and by the maximality of $B$, we have $a
\in B$; whence $C(Q)\subseteq B$. Consequently, $P\cap C(Q)\subseteq P
\cap B=Q$, so $Q=P\cap C(Q)$, and thus $Q$ is a block in $P$. Suppose $Q
\sb{1}$ is a second block in $P$ such that $Q\sb{1}\subseteq B$. Then
$Q\sb{1}\subseteq B\cap P=Q$ and it follows that $Q\sb{1}=Q$.
\end{proof}

\begin{lemma} \label{lm:supinblock}
Let $Q$ be a block in $P$, let $V\subseteq Q$, let $B$ be a C-block in $A$
and let $T\subseteq B$. Then{\rm: (i)} If $V$ has a supremum $p$ in $P$,
then $p\in Q$ and $p$ is the supremum of $V$ in $Q$. Also, if $V$ has a
supremum $a$ in $A$, then $a\in Q$ and $a$ is the supremum of $V$ both in
$Q$ and in $P$. {\rm(ii)} If $T$ has a supremum $b$ in $A$, then $b\in B$
and $b$ is the supremum of $T$ in $B$. {\rm(iii)} If $P$ is a complete
OML, then $P\cap B$ is a complete Boolean algebra.
\end{lemma}

\begin{proof}
(i) Since $Q$ is commutative, we have $Q\subseteq C(Q)$, whence $CC(Q)
\subseteq C(Q)$. Thus, by Theorem 6.3 (i), $p\in CC(V)\subseteq
CC(Q)\subseteq C(Q)$, and it follows from the maximality of $Q$
that $p\in Q$, whence $p$ is the supremum of $V$ in $Q$. Similarly,
by Theorem \ref{th:CC(T)}, $a\in CC(V)\subseteq C(Q)$, so $a\in Q
\subseteq P$, and $a$ is the supremum of $V$ both in $Q$ and in $P$.

(ii) By Theorem \ref{th:CC(T)}, $b\in CC(T)\subseteq CC(B)=B$.

(iii) Suppose that $P$ is a complete OML. By Theorem \ref{th:uniqueblock},
$Q:=P\cap B$ is a block in $P$, and by (i), $Q$ is a complete Boolean
algebra.
\end{proof}

\begin{theorem} \label{th:BanachSAprops}
If $A$ is a Banach synaptic algebra, then every C-block $B$ in $A$ is a
commutative Banach synaptic algebra that is both Dedekind and monotone
$\sigma$-complete and the Boolean algebra $B\cap P$ of projections in $B$
is $\sigma$-complete.
\end{theorem}

\begin{proof} Let $B$ be a C-block in $A$. Then $B=C(B)$ is norm-closed,
and since $A$ is Banach, so is $B$. Therefore $B$ is a commutative Banach
synaptic algebra, whence the theorem follows from Theorem
\ref{th:ComBanachsigma}
\end{proof}

\begin{theorem} \label{th:Pcomplete}
{\rm(i)} If every C-block in $A$ is Dedekind complete {\rm(}Dedekind
$\sigma$-complete{\rm)}, then every block in $P$ is a complete {\rm(}a
$\sigma$-complete{\rm)} Boolean algebra.
{\rm(ii)} Every block in $P$ is a complete {\rm(}a $\sigma$-complete{\rm)}
Boolean algebra iff $P$ is a complete {\rm(}a $\sigma$-complete{\rm)} OML.
\end{theorem}

\begin{proof}
(i) Assume that every C-block in $A$ is Dedekind complete and that $Q$
is a block in $P$. By Theorem \ref{th:uniqueblock}, there is a unique
C-block $B$ in $A$ with $Q\subseteq B$ and $Q=P\cap B$ is the Boolean
algebra of projections in $B$. Thus, by Theorem \ref{th:Dedekind&monotone}
with $A$ replaced by $B$, it follows that $Q$ is a complete Boolean algebra.
The proof for the Dedekind $\sigma$-complete case is similar.

(ii) If $P$ is complete, then so is every block in $P$ by Lemma
\ref{lm:supinblock} (i). Conversely, suppose that every block in $P$ is
complete. To prove that $P$ is complete, it will be sufficient to prove
that it is orthocomplete (Theorem \ref{th:Holland} (i)), so let $D$ be
an orthogonal subset of $P$. Note that $D$ is a commutative subset of
$P$. We have to prove that $D$ has a supremum in $P$. Without loss of
generality, we can assume that $0\in D$. If $D$ is a maximal orthogonal
set, then $1$ is the supremum of $D$ in $P$ (Theorem \ref{th:Holland}
(ii)), so we can assume that $D$ is not a maximal orthogonal set. Therefore
the set $V:=\{v\in P:v\not=0\text{\ and\ }d\in D\Rightarrow v\perp d\}$
is nonempty. By Zorn's lemma, there is a maximal commutative subset
$V\sb{1}$ of $V$. For $v\in V\sb{1}$ and $d\in D$, we have $v\perp d$,
so $vCd$, and it follows that $V\sb{1}\cup D$ is a commutative subset of $P$
and that $d\sp{\perp}$ is an upper bound for $V\sb{1}$ in $P$. By Zorn's lemma
again, there is a block $Q$ in $P$ with $V\sb{1}\cup D\subseteq Q$.

By hypothesis, $Q$ is complete; hence, $V\sb{1}$ has a supremum $p$ in $Q$.
For $d\in D$, we have $d\in Q$, whence $d\sp{\perp}\in Q$, so $d\sp{\perp}$
is an upper bound for $V\sb{1}$ in $Q$. Therefore, $p\leq d\sp{\perp}$, so
$d\leq p\sp{\perp}$, whereupon $p\sp{\perp}$ is an upper bound for $D$ in $Q$.
We claim that, in fact, $p\sp{\perp}$ is the supremum of $D$ in $P$, which
will complete the proof.

Thus, suppose that $b$ is an upper bound for $D$ in $P$ and put $c:=
p\sp{\perp}\wedge(p\vee b\sp{\perp})$. Then, $c\leq p\sp{\perp}$, so
$p\leq c\sp{\perp}$, and $v\in V\sb{1}\Rightarrow v\leq p$, whence $v
\in V\sb{1}\Rightarrow v\leq c\sp{\perp}$. In particular, $c\in C(V\sb{1})$.
Also, $d\in D\Rightarrow d\leq p\sp{\perp}\wedge b\leq p\vee(p\sp{\perp}
\wedge b)=c\sp{\perp}\Rightarrow c\perp d$.  If $c\not=0$, it follows that
$c\in V\sb{1}$, whence $c\leq c\sp{\perp}$, and so $c=0$, a contradiction.
Therefore, $c=p\sp{\perp}\wedge(p\vee b\sp{\perp})=0$, and since $p$, hence
also $p\sp{\perp}$ commutes with $p\vee b\sp{\perp}$, we have $p\sp{\perp}
\leq(p\vee b\sp{\perp})\sp{\perp}=p\sp{\perp}\wedge b\leq b$. Thus, $P$ is
complete. A similar argument takes care of the $\sigma$-complete case.
\end{proof}

\begin{corollary} \label{co:everyBsigmacom}
If every C-block in $A$ is Dedekind complete {\rm(}Dedekind
$\sigma$-complete{\rm)}, then $P$ is a complete {\rm(}a
$\sigma$-complete{\rm)} OML.
\end{corollary}

\section{The self-adjoint part of an AW$\sp{\ast}$-algebra} \label{sc:AW*}

In this section, we present some conditions that are equivalent to the
requirement that the synaptic algebra $A$ is isomorphic to the self-adjoint
part of an AW$\sp{\ast}$-algebra.

I. Kaplansky introduced AW$\sp{\ast}$-algebras in \cite{Kap} as algebraic
generalizations of W$\sp{\ast}$ (i.e., von Neumann) algebras. According to
Kaplansky's original definition, an AW$\sp{\ast}$-algebra is a
C$\sp{\ast}$-algebra ${\mathcal C}$ such that (1) the OML of projections in
${\mathcal C}$ is orthocomplete and (2) any maximal commutative
$\ast$-subalgebra of ${\mathcal C}$ is norm-generated by its own projections.
Nowadays, the following equivalent definition \cite[p. 853]{KapMod} (which we
shall use) is often given. In the definition $Sc:=\{sc:s\in S\}$ and
$p\sb{S}{\mathcal C}:=\{p\sb{S}a:a\in{\mathcal C}\}$.

\begin{definition} \label{df:AW*}
An \emph{AW$\sp{\ast}$-algebra} is a C$\sp{\ast}$-algebra ${\mathcal C}$
that is a Baer$\sp{\ast}$ ring {\rm\cite{SKB}}, i.e., for every subset
$S\subseteq{\mathcal C}$, there is a projection $p\sb{S}\in{\mathcal C}$
such that the right annihilator of $S$, namely $\{c\in{\mathcal C}:Sc=
\{0\}\}$, is the principal right ideal $p\sb{S}{\mathcal C}$ of ${\mathcal C}$
generated by the projection $p\sb{S}$.
\end{definition}

If ${\mathcal C}$ is an AW$\sp{\ast}$-algebra, then it is a Rickart
C$\sp{\ast}$-algebra; indeed, if $c\in{\mathcal C}$, take $S:=\{c\}$,
and we have $ca=0\Leftrightarrow a=p\sb{S}a$ for all $a\in A$.
Therefore (Theorem \ref{th:RBSA}), the self-adjoint part of an
AW$\sp{\ast}$-algebra is a Banach synaptic algebra.

\begin{definition} \label{df:BaerCompleteCarrier}
In this definition, if $T\subseteq A$, $a\in A$ and $p\in P$, then $Ta
:=\{ta:t\in T\}$ and $pA:=\{pa:a\in A\}$ for $p\in P$.
\begin{enumerate}
\item[(1)] The synaptic algebra $A$ has the \emph{Baer property} iff, for every
subset $T\subseteq A$, there is a projection $p\sb{T}\in P$ such that $\{a\in A:
Ta=\{0\}\}=p\sb{T}A$.
\item[(2)] $A$ has the \emph{complete carrier property} iff, for every
$T\subseteq A$, there is a projection $q\sb{T}\in P$ such that, for all
$a\in A$, $Ta=\{0\}\Leftrightarrow q\sb{T}a=0$.
\end{enumerate}
\end{definition}

\begin{theorem} \label{th:Baer,CC,completeP}
The following conditions are mutually equivalent{\rm:}
\begin{enumerate}
\item $A$ has the Baer property.
\item $A$ has the complete carrier property.
\item $P$ is a complete OML.
\end{enumerate}
\end{theorem}

\begin{proof}
(i) $\Leftrightarrow$ (ii). With the notation of Definition
\ref{df:BaerCompleteCarrier}, if $T\subseteq A$ and $A$ has the Baer
property, put $q\sb{T}:=1-p\sb{T}$, and if $A$ has the complete carrier
property, put $p\sb{T}:=1-q\sb{T}$.

\medskip

(ii) $\Rightarrow$ (iii). Assume (ii) and let $V\subseteq P$. Then there
is a projection $q\sb{V}\in P$ such that, for all $a\in A$, $Va=\{0\}
\Leftrightarrow q\sb{V}a=0$, and by Lemma \ref{lm:V}, $q\sb{V}$ is the
supremum of $V$ in the OML $P$. By duality, every subset of the OML $P$
has an infimum.

\medskip

(iii) $\Rightarrow$ (ii). Assume (iii), let $T\subseteq A$, put $S:=
\{t\dg:t\in T\}\subseteq P$, and let $r$ be the supremum of $S$ in
$P$.  By Theorem \ref{th:SupContinuity} (ii), for every $a\in A$, $Sa
=0\Leftrightarrow  ra=0$. Now $ra=0\Leftrightarrow t\dg a=0, \forall t
\in T\Leftrightarrow ta=0,\forall t\in T\Leftrightarrow Ta=\{0\}$, whence
$A$ satisfies the condition in Definition \ref{df:BaerCompleteCarrier}
(2) with $q\sb{T}=r$.
\end{proof}

\begin{theorem} \label{th:Baer}
Let ${\mathcal C}$ be a Rickart C$\sp{\ast}$ algebra and
organize the self-adjoint part ${\mathcal C}\sp{sa}$ of
${\mathcal C}$ into a Banach synaptic algebra {\rm(Theorem
\ref{th:RBSA})}. Then ${\mathcal C}$ is an AW$\sp{\ast}$-algebra
iff ${\mathcal C}\sp{sa}$ has the Baer property.
\end{theorem}

\begin{proof} Suppose that ${\mathcal C}$ is an AW$\sp{\ast}$-algebra
and, $T\subseteq{\mathcal C}\sp{sa}$, and let $a\in{\mathcal C}\sp{sa}$.
Then $T\subseteq{\mathcal C}$ and $a\in C$,  whence, there is a
projection $p\sb{T}$ in ${\mathcal C}\sp{sa}$ such that, $Ta=\{0\}
\Leftrightarrow a=p\sb{T}a$. Therefore ${\mathcal C}\sp{sa}$ has the
Baer property.

Conversely, suppose that ${\mathcal C}\sp{sa}$ has the Baer property,
$S\subseteq{\mathcal C}$, and $c\in{\mathcal C}$. Then $T:=\{s
\sp{\ast}s:s\in S\}\subseteq{\mathcal C}\sp{sa}$, so there exists a
projection $p\sb{T}\in{\mathcal C}\sp{sa}$ such that, for all $a\in
{\mathcal C}\sp{sa}$, $Ta=\{0\}\Leftrightarrow a=p\sb{T}a$. Putting
$a=cc\sp{\ast}$, we obtain $Tcc\sp{\ast}=\{0\}\Leftrightarrow
cc\sp{\ast}=p\sb{T}cc\sp{\ast}$, whence by Lemma \ref{lm:cc*} (ii),
for all $c\in{\mathcal C}$,
\begin{equation} \label{eq:1}
s\sp{\ast}scc\sp{\ast}=0\text{\ for all\ }s\in S\Leftrightarrow c=
p\sb{T}c.
\end{equation}
To prove that ${\mathcal C}$ is an AW$\sp{\ast}$-algebra, it will be
sufficient to prove that $sc=0\text{\ for all\ }s\in S\Leftrightarrow c
=p\sb{T}c$. If $sc=0$ for all $s\in S$, then $s\sp{\ast}scc\sp{\ast}=0$
for all $s\in S$, so $c=p\sb{T}c$ by (\ref{eq:1}). Conversely, suppose
that $c=p\sb{T}c$. Putting $c=p\sb{T}$ in (\ref{eq:1}) and observing
that $p\sb{T}=p\sb{T}p\sb{T}$, we find that, for all $s\in S$, $s
\sp{\ast}sp\sb{T}=s\sp{\ast}sp\sb{T}p\sb{T}\sp{\ast}=0$, whence
$(sp\sb{T})\sp{\ast}sp\sb{T}=p\sb{T}s\sp{\ast}sp\sb{T}=0$, and it
follows that $sp\sb{T}=0$. Therefore, $sc=sp\sb{T}c=0$.
\end{proof}

\begin{theorem} \label{th:AW*conditions}
The synaptic algebra $A$ is isomorphic to the self-adjoint part of an
AW$\sp{\ast}$-algebra iff it is Banach and satisfies any one {\rm(}hence
all{\rm)} of the following equivalent conditions{\rm: (i)} $A$ has the
Baer property. {\rm (ii)} $A$ has the complete carrier property.
{\rm (iii)} The OML $P$ of projections in $A$ is complete.
\end{theorem}

\begin{proof}
Conditions (i)--(iii) are mutually equivalent by Theorem \ref
{th:Baer,CC,completeP}. Suppose that $A$ is isomorphic to the self-adjoint
part ${\mathcal C}\sp{sa}$ of an AW$\sp{\ast}$-algebra ${\mathcal C}$.
Than $A$ is Banach and by Theorem \ref{th:Baer}, ${\mathcal C}\sp{sa}$,
hence also $A$, has the Baer property. Conversely, suppose that $A$ is
Banach and satisfies any one of the equivalent conditions (i), (ii),
or (iii). Then it satisfies (i), and since it is Banach, it is
isomorphic to the self-adjoint part ${\mathcal C}\sp{sa}$ of a
Rickart C$\sp{\ast}$-algebra ${\mathcal C}$ by Theorem \ref{th:RBSA}.
But then, ${\mathcal C}\sp{sa}$ has the Baer property, so ${\mathcal
C}$ is an AW$\sp{\ast}$-algebra by Theorem \ref{th:Baer}.
\end{proof}

In \cite{SaWr}, K. Sait\^o and J.D.M. Wright proved that a
C$\sp{\ast}$-algebra ${\mathcal C}$ is an AW$\sp{\ast}$-algebra
iff every maximal abelian ${\ast}$-subalgebra of ${\mathcal C}$ is
monotone complete. The equivalence of parts (i) and (v) in the
following theorem can be regarded as an analogue for synaptic
algebras of the Sait\^o-Wright theorem.

\begin{theorem} \label{th:ABanachPcomplete}
Let $A$ be a Banach synaptic algebra. Then the following conditions are
mutually equivalent{\rm: (i)} Every C-block in $A$ is monotone complete.
{\rm(ii)} Every C-block in $A$ is Dedekind complete. {\rm(iii)} Every
block in $P$ is a complete Boolean algebra. {\rm(iv)} $P$ is a complete OML.
{\rm(v)} $A$ is isomorphic to the self-adjoint part of an AW$\sp{\ast}$-algebra
\end{theorem}

\begin{proof}
That (i) $\Leftrightarrow$ (ii) follows from Theorem \ref{th:Dedekind&monotone}
(i). By Theorem \ref{th:Pcomplete}, we have (ii) $\Rightarrow$ (iii)
$\Leftrightarrow$ (iv). We claim that (iv) $\Rightarrow$ (i). Thus, assume that
$A$ is Banach, $P$ is a complete OML, and $B$ is a C-block in $A$. Then $B$ is
a norm-closed, hence Banach, commutative synaptic algebra under the restrictions
of the partial order and operations on $A$, and $P\cap B$ is the Boolean algebra of projections in $B$. By Lemma \ref{lm:supinblock} (iii), $P\cap B$ is a complete
Boolean algebra, and, applying Theorem \ref{th:ComBanach} to the commutative Banach
synaptic algebra $B$, we infer that $B$ is monotone complete. Thus, conditions
(i)--(iv) are mutually equivalent, and by Theorem \ref{th:AW*conditions}, (iv)
$\Leftrightarrow$ (v).
\end{proof}

\section{Synaptic algebras and GH-algebras} \label{sc:SA&GH}

Axioms for a \emph{generalized Hermitian algebra} (GH-algebra) can be
found in \cite[Definition 2.1]{GHAlg1}. By the discussion in \cite
[\S 6]{FSyn} and Theorem \ref{th:CC(T)}, we have the following.

\begin{theorem}\label{th:GH}
A GH-algebra is the same thing as a synaptic algebra $A$ such that
every bounded monotone increasing sequence $a\sb{1}\leq a\sb{2}
\leq\cdots$ of pairwise commuting elements in $A$ has a supremum
in $A$.
\end{theorem}

\begin{corollary} \label{co:MsigmaComisGH}
If $A$ is monotone $\sigma$-complete, then $A$ is a Banach GH-algebra.
\end{corollary}

\begin{proof}
If $A$ is monotone $\sigma$-complete, then $A$ is a GH-algebra by
Theorem \ref{th:GH} and $A$ is Banach by \cite[Theorem 2.4]{FJPLS}.
\end{proof}

\begin{theorem}\label{th:basacom}
Suppose that $A$ is a commutative synaptic algebra. Then the following
conditions are mutually equivalent{\rm: (i)} $A$ is Dedekind
$\sigma$-complete. {\rm(ii)} $A$ is monotone $\sigma$-complete. {\rm(iii)}
$A$ is Banach.  {\rm(iv)} $A$ is a GH-algebra.
\end{theorem}

\begin{proof} Obviously (i) $\Rightarrow$ (ii) and by \cite[Theorem 2.4]{FJPLS},
(ii) $\Rightarrow$ (iii). By Theorem \ref{th:ComBanachsigma}, (iii)
$\Rightarrow$(i), whence (i), (ii), and (iii) are mutually equivalent. Also,
Since $A$ is commutative, the equivalence (ii) $\Leftrightarrow$ (iv) is
obvious from Theorem \ref{th:GH}.
\end{proof}

In the following corollary of Theorem \ref{th:basacom}, note that every
C-block $B\subseteq A$ and every bicommutant $CC(M)$ of a commutative
subset $M\subseteq A$ qualifies as a norm-closed commutative sub-synaptic
algebra of $A$.

\begin{corollary} \label{co:basa}
Let $A$ be a Banach synaptic algebra and let $H$ be a norm-closed commutative
sub-synaptic algebra of $A$. Then $H$ is a Banach GH-algebra.
\end{corollary}

\begin{proof} As  $A$ is norm complete and $H$ is norm-closed, it follows
that $H$ is norm-complete. Therefore, we can apply Theorem \ref{th:basacom}
to the commutative synaptic algebra $H$ and conclude that $H$ is a GH-algebra.
\end{proof}

\begin{theorem} \label{th:GHCblock}
If $A$ is a GH-algebra, then every C-block in $A$ is a commutative monotone
$\sigma$-complete Banach GH-algebra and $P$ is a $\sigma$-complete OML.
\end{theorem}

\begin{proof}
Suppose $A$ is a GH-algebra, let $B$ be a C-block of $A$, and let
$(b\sb{n})\sb{{n\in\Nat}}$ be a bounded monotone increasing sequence
in the commutative synaptic algebra $B$. Then $(b\sb{n})\sb{{n\in\Nat}}$
is a bounded monotone increasing sequence of pairwise commuting elements in
$A$, so it has a supremum $a\in A$, whence by Theorem \ref{th:CC(T)},
$a\in CC((b\sb{n})\sb{n\in\Nat})\subseteq CC(B)=B$. Thus, $B$ is a
commutative monotone $\sigma$-complete synaptic algebra, hence it is a
GH-algebra. By Corollary \ref{co:everyBsigmacom}, $P$ is $\sigma$-complete.
\end{proof}

\begin{remark}
Even if every C-block $B$ in the synaptic algebra $A$ is monotone
$\sigma$-complete, we cannot conclude that $A$ is a GH-algebra because the
supremum in a C-block $B\subseteq A$ of a sequence $(b\sb{n})\sb{{n\in\Nat}}$
in $B$ is not necessarily the supremum of the sequence in $A$.
\end{remark}

We conclude this section by reviewing two functional representation theorems
for a commutative GH-algebra $A$. By Theorem \ref{th:basacom}, $A$
is a Dedekind $\sigma$-complete Banach synaptic algebra and $P$ is a
$\sigma$-complete Boolean algebra. Thus, as is well-known, the Stone space
$X$ of $P$ is not only totally disconnected, but it is basically disconnected,
i.e., the closure of any open F$\sb{\sigma}$ subset of $X$ remains open.
As in the proof of Theorem \ref{th:ComBanachsigma}, the lattice ordered
commutative associative unital Banach algebra $C(X,\reals)$ with pointwise
partial order and pointwise operations, and with the supremum norm, of all
continuous real-valued functions on $X$ is a GH-algebra; moreover, $A$ is
isomorphic (as a synaptic algebra) to $C(X,\reals)$. A disadvantage of this
functional representation is that the supremum of an increasing sequence in
$C(X,\reals)$ is not necessarily the pointwise supremum.

A second functional representation of the commutative GH-algebra $A$
avoids the disadvantage mentioned above. This representation involves a
so-called gh-tribe on the Stone space $X$ of $P$. By \cite[Definition 6.3]
{FJPLS}, a \emph{gh-tribe} on $X$ is a set ${\mathcal T}$ of bounded
real-valued functions on $X$ such that: (1) the constant functions 0 and 1
belong to ${\mathcal T}$; (2) ${\mathcal T}$ is closed under pointwise sums
and real multiples of its elements; and (3) ${\mathcal T}$ is closed under
pointwise suprema of sequences of its elements which are bounded above by an
element of ${\mathcal T}$. It turns out that such a gh-tribe ${\mathcal T}$
is closed under pointwise multiplication (\cite[Lemma 6.4]{FJPLS}) and it
is a GH-algebra. Moreover, by \cite[Theorem 6.6]{FJPLS}, there exists a gh-tribe
${\mathcal T}$ on $X$ such that $C(X,\reals)\subseteq{\mathcal T}$ and there
exists a surjective synaptic morphism $h$ of GH-algebras from ${\mathcal T}$
onto $A$. The morphism $h$ has the property that, for $f,g\in{\mathcal T}$,
$h(f)=h(g)$ iff $\{x\in X:f(x)\not=g(x)\}$ is a meager subset of $X$, i.e.,
it is a countable union of nowhere dense sets. This representation by a
gh-tribe can be regarded as an analogue of the classical Loomis-Sikorski
theorem for $\sigma$-complete Boolean algebras.

\end{document}